\documentclass[a4paper,12pt]{amsart}

\usepackage[margin=2.0cm]{geometry}
\usepackage{amsthm,amsmath,amssymb}
\usepackage{mathtools}
\usepackage[ruled,linesnumbered]{algorithm2e}
\usepackage[utf8]{inputenc}
\usepackage[english]{babel}
\usepackage[linktocpage]{hyperref}
\usepackage{xcolor}
\hypersetup{colorlinks=true,linkcolor=blue,citecolor=blue,filecolor=magenta,urlcolor=blue}
\usepackage{tikz, tikz-cd} \usetikzlibrary{arrows.meta}

\theoremstyle{plain}
\newtheorem{theorem}{Theorem}[section]
\newtheorem{prop}[theorem]{Proposition}

\newtheorem{lemma}[theorem]{Lemma}

\theoremstyle{definition}
\newtheorem{defi}[theorem]{Definition}
\newtheorem{ex}[theorem]{Example}

\theoremstyle{remark}
\newtheorem{remark}[theorem]{Remark}
\newtheorem{notation}[theorem]{Notation}

\def\CC{\mathbb{C}}

\def\ZZ{\mathbb{Z}}
\def\QQ{\mathbb{Q}}

\DeclareMathOperator{\Ann}{Ann}
\DeclareMathOperator{\Sat}{Sat}
\DeclareMathOperator{\syz}{syz}
\DeclareMathOperator{\lcm}{lcm}

\title[Annihilator and B-S polynomial of a rational function]
{An algorithm for annihilator and Bernstein-Sato polynomial of a rational function}

\author[M.~González-Villa]{Manuel González-Villa}
\address[M.~González-Villa]{Centro de investigación en Matemáticas, Apartado Postal 402, 36000 Guanajuato, Gto., México}
\urladdr{\url{https://www.cimat.mx/~manuel.gonzalez/}}
\email{\href{mailto:manuel.gonzalez@cimat.mx}{manuel.gonzalez@cimat.mx}} 

\author[E.~León-Cardenal]{Edwin León-Cardenal}
\address[E.~León-Cardenal]{Departamento de Matemáticas y Computación, Universidad de La Rioja\\
C. Madre de Dios, 53, 26006, Logro\~no, Spain.}
\urladdr{\url{http://riemann.unizar.es/~eleon}}
\email{\href{mailto:edwin.leon@unirioja.es}{edwin.leon@unirioja.es}}

\author[V.~Levandovskyy]{Viktor Levandovskyy}
\address[V.~Levandovskyy]{EPAM School of Digital Technologies \\
American University Kyiv, Ukraine} 
\urladdr{\url{http://www.math.rwth-aachen.de/homes/Viktor.Levandovskyy/}}
\email{\href{mailto:viktor.levandovskyy@auk.edu.ua}{viktor.levandovskyy@auk.edu.ua}}
\author[J.~Martín-Morales]{Jorge Martín-Morales}
\address[J.~Martín-Morales]{Departamento de Matemáticas, IUMA \\
Universidad de Zaragoza \\ 
C.~Pedro Cerbuna 12 \\ 
50009 Zaragoza, Spain} 
\urladdr{\url{http://riemann.unizar.es/~jorge}}
\email{\href{mailto:jorge@unizar.es}{jorge.martin@unizar.es}}

\subjclass[2020]{14F10, 16S32, 68W30, 32S40}  
\keywords{Rational function, Bernstein-Sato polynomial, annihilator, saturation, Ore localization, noncommutative Gröbner basis, monodromy}
\thanks{\noindent 
JM is partially supported by the European Union NextGenerationEU/PRTR and by MICIU/AEI/10.13039/501100011033 (grant codes: RYC2021-034300-I
and CNS2024-154271).
MG, EL, and JM are partially supported by MICIU/AEI/10.13039/501100011033 (grant code: PID2020-114750GB-C31)
and by Departamento de Ciencia, Universidad y Sociedad del Conocimiento del Gobierno de Arag{\'o}n 
(grant code: E22\_20R: ``{\'A}lgebra y Geometr{\'i}a'').
JM is also supported by Junta de Andaluc{\'\i}a (FQM-333). MG and EL are also partially supported by the European Union NextGenerationEU/PRTR and UNIZAR via María Zambrano's Program and by CONAHCYT project CF-2023-G33.}

\begin{document}

\begin{abstract}
The singularity theory of rational functions, i.e., the quotient of two polynomials, has been investigated
in the past two decades. The Bernstein-Sato polynomial of a rational function has recently been introduced by Takeuchi. However, only trivial examples are known. We provide an algorithm for computing the Bernstein-Sato polynomial in this context. The strategy is to compute the annihilator of the rational function by using the annihilator of the pair consisting of the numerator and denominator of the quotient. In a natural way a non-vanishing condition
on the Bernstein-Sato ideal of the pair appears. This method has been implemented in freely available computer algebra system \textsc{Singular}. It relies on Gr\"obner bases in noncommutative PBW algebras. The algorithm allows us to exhibit some explicit non-trivial examples and to support some existing conjectures.
\end{abstract}

\maketitle
%\tableofcontents

\section{Introduction}

The theory of $b$-functions was developed by M.~Sato during the 60's under the name of algebraic theory of linear differential systems \cite{Sat90}. The main purpose was to study prehomogeneous vector spaces and their zeta functions \cite{SaShi74,SKKO80}. A similar development was taking form a few years later with the study of the properties of the ring of differential operators by I.~N.~Bernstein~\cite{Ber71,Ber72}. The whole mathematical construction is known today as $\mathcal{D}$-module theory and it has connections in many areas of mathematics such as algebraic geometry, singularity theory, topology of varieties, representation theory and differential equations, among others. For a nice introduction to $\mathcal{D}$-module theory we refer to~\cite{Cou95}.

Denote by $D$ the $n$th Weyl algebra over $\CC$, i.e., the ring of differential operators with polynomial coefficients in $n$ variables over $\CC$ or 
\[
\CC [x] \langle\partial\rangle=\CC [x_1,\ldots,x_n] \langle \partial_1,\ldots,\partial_n\rangle,
\] 
where the variables obey the following noncommutative relations $\partial_i x_i = x_i \partial_i + 1$, $i=1,\ldots,n$.
Pick another variable $s$ and consider the algebra $D[s]=D\otimes_\CC\CC[s]$. Given a non-constant polynomial $f \in \CC[x]$ we denote by $\CC[x,s,\frac{1}{f}]$ the localization of $\CC[x,s]$ with respect to the multiplicative set $\{f^k\}_{k\in \ZZ_{\geq 0}}$.
The free $\CC[x,s,\frac{1}{f}]$-module of rank one generated by $f^s$, $M=\CC[x,s,\frac{1}{f}] f^s$, has a natural structure
of $D[s]$-module.
A key result \cite{Ber72} in the theory of $\mathcal{D}$-modules asserts that there exists an element $P(s) \in D[s]$ and a nonzero polynomial $b(s)\in \CC[s]$ verifying the following functional equation 
\begin{equation}\label{Eq:Func_Eq}
	P(s)f^{s+1} = b(s)f^s.
\end{equation}
The set of all polynomials $b(s)$ of $\CC[s]$ satisfying \eqref{Eq:Func_Eq} forms an ideal, and its unique monic generator is called the \textit{Bernstein-Sato polynomial} $b_f(s)$ \textit{of} $f$.
The roots of $b_f(s)$ encode invariants of the singularities of $f$. Kashiwara showed that they are negative rational
numbers~\cite{Kas76}. Malgrange introduced the $V$-filtration and proved that if $s_0$ is a root of $b_f(s)$, then $e^{-2\pi is_0}$ is an eigenvalue of the complex monodromy of $f$ \cite{Mal74,Mal83}.

A more general functional equation that will be relevant in this work involves several functions. 
Given $p$ polynomials $f_1,\ldots,f_p$ in $\CC[x_1,\ldots,x_n]$,
Sabbah \cite{Sab87} showed that there exists a nonzero polynomial $b(s_1,\ldots,s_p) \in \CC[s_1,\ldots,s_p]$ such that
\begin{equation}\label{Eq:Func_Eq_Ideal}
	P(s_1,\ldots,s_p)f_1^{s_1+1}\cdots f_p^{s_p+1} = b(s_1,\ldots,s_p)f_1^{s_1}\cdots f_p^{s_p},
\end{equation}
with $P(s_1,\ldots,s_p)\in D[s_1,\ldots,s_p] = D \otimes_\CC \CC[s_1,\ldots,s_p]$.
The set of polynomials satisfying \eqref{Eq:Func_Eq_Ideal} forms an ideal in $\CC[s_1,\ldots,s_p]$ that we refer to as the \textit{Bernstein-Sato ideal} of $f_1,\ldots,f_p$ and it is denoted by
$B_{f_1,\ldots,f_p}(s_1,\ldots,s_p)$.
For an extensive discussion of the connections between $b$-functions and the corresponding singularities, see for instance~\cite{Po18, AJNB21}.

On the computational side, Oaku~\cite{Oak97} designed the first general algorithm for computing the annihilator of $f^s$
and obtain $b_f(s)$. The use of Gr\"obner deformations and initial forms to study Bernstein-Sato polynomials is explained by Saito,
Sturmfels, and Takayama~\cite{SST00}. Then Noro~\cite{Nor02} provided 
%an efficient 
a modular and linear-algebra-driven algorithm for computing the global $b$-function. A family of algorithms for checking candidate roots of $b_f(s)$ is presented by Levandovskyy
and Martín-Morales~\cite{LM12}. For an overview of other computational contributions, see e.g.~\cite{Wal15}.
Some of these methods have been implemented in different computer algebra systems \cite{OpenXM, LeTs00, LM10, AL12}.

Recently Bernstein-Sato polynomial for rational functions was introduced by Takeuchi \cite{Tak23}.
The Bernstein-Sato polynomial of $\frac{f}{g}$, related to the given nonnegative integer $m$, is the monic polynomial of smallest degree satisfying the 
%following 
relation
\begin{equation*}%\label{Eq:FuncEqTake}
b_{\frac{f}{g},m}^{}(s) \frac{1}{g^m} \left( \frac{f}{g} \right)^s \in
\sum_{k = 1}^\infty D[s] \frac{1}{g^m} \left( \frac{f}{g} \right)^{s+k}
% \sum_{k = 1}^N P_k(s) \frac{1}{g^m} \left( \frac{f}{g} \right)^{s+k} = b_{\frac{f}{g},m}^{}(s) \frac{1}{g^m} \left( \frac{f}{g} \right)^s.
\end{equation*}
in the free $\CC[x,s,\frac{1}{fg}]$-module of rank one $\CC[x,s,\frac{1}{fg}] \left( \frac{f}{g} \right)^s$.
A similar approach was proposed by Àlvarez Montaner, González Villa, León-Cardenal and Núñez-Betancourt in \cite{AGLN21}. The polynomial $b_{\frac{f}{g},m}(s)$ is shown to satisfy
some expected relations with respect to other invariants of the singularity of rational functions.
However, only trivial examples are known.

In this work we provide an algorithm to compute $b_{\frac{f}{g},m}^{(N)}(s)$
defined by the functional equation
\begin{equation}\label{eq:functional-N}
b_{\frac{f}{g},m}^{(N)}(s) \frac{1}{g^m} \left( \frac{f}{g} \right)^s =
\sum_{k = 1}^N P_k(s) \frac{1}{g^m} \left( \frac{f}{g} \right)^{s+k},
\end{equation}
where $P_k(s) \in D[s]$. Note that $b_{\frac{f}{g},m}(s) = b_{\frac{f}{g},m}^{(N)}(s)$ when $N$ is big enough.
In Section~\ref{sec:examples} we provide explicit examples of the Bernstein-Sato polynomial of rational functions.
Following Oaku's original ideas \cite{Oak97} we first compute the annihilator in $D[s]$
of $\frac{1}{g^m} \left( \frac{f}{g} \right)^s$ denoted by
\begin{equation}\label{eq:ann-mero}
\Ann_{D[s]} \frac{1}{g^m} \left( \frac{f}{g} \right)^s.
\end{equation}
Towards the computation of \eqref{eq:ann-mero}, we first compute the annihilator in $D[s_1,s_2]$ of
$f^{s_1} g^{s_2}$, see e.g.~\cite{OaTa99}, that we denote by $I(s_1,s_2)$.
After the substitution $s_1 = s$ and $s_2 = -s-m$ we obtain a new ideal  $I(s,-s-m)$ in $D[s]$ that it is in general
not $\CC[s]$-saturated. Hence $I(s,-s-m)$ does not agree with the one in~\eqref{eq:ann-mero}.
Theorem~\ref{thm:main-theorem} gives a sufficient non-vanishing condition on the Bernstein-Sato ideal of the pair $(f,g)$
to recover the annihilator of $\frac{1}{g^m} \left( \frac{f}{g} \right)^s$ as the saturation of $I(s,-s-m)$.
Even if this condition is not satisfied a careful analysis of the factors of the form $s_1+s_2+\ell$ with $\ell \in \ZZ$
allows us to compute the annihilator~\eqref{eq:ann-mero}, see Section~\ref{sec:L-discussion}.
These algorithms have been implemented in \textsc{Singular} to exhibit explicit examples.

The paper is organized as follows. In Section \ref{sec:settings}, we give the necessary definitions and notations for the concepts
to be dealt with in this paper such as Bernstein-Sato polynomial of rational functions and their related modules. In Section \ref{sec:sat},
we recall a method for computing the saturation of an ideal in a noncommutative context.
The main result of our work, namely Theorem~\ref{thm:main-theorem}, is presented in Section~\ref{sec:annhilator}, and this allowed us
to describe an algorithm to compute $b_{\frac{f}{g}}^{(N)}(s)$ in Section \ref{sec:BS}.
Finally, Section \ref{sec:examples} is devoted to showing some examples.

\section{Settings}\label{sec:settings}

Let $f,g \in \CC[x] = \CC[x_1,\ldots,x_n]$ be two non-constant polynomials and choose another variable~$s$.
We denote by $\CC[x,s,\frac{1}{fg}]$ the localization of $\CC[x,s]$ with respect to the
multiplicatively closed set $\{(fg)^k\}_{k\in \ZZ_{\geq 0}}$. Consider the free $\CC[x,s,\frac{1}{fg}]$-module
of rank one
\[
M = \CC\left[x,s,\frac{1}{fg}\right] \left( \frac{f}{g} \right)^s.
\]
Note that $M$ has a natural structure of $D[s]$-module, where $D[s] = D \otimes_\CC \CC[s]$ and $D$ is the $n$th Weyl algebra,
as follows
\[
\begin{aligned}
x_i \bullet \frac{h(s)}{(fg)^k} \left(\frac{f}{g}\right)^s &= \frac{x_i h(s)}{(fg)^k}  \left(\frac{f}{g}\right)^s, \\
s \bullet \frac{h(s)}{(fg)^k} \left(\frac{f}{g}\right)^s &= \frac{s h(s)}{(fg)^k} \left(\frac{f}{g}\right)^s, \\
\partial_i \bullet \frac{h(s)}{(fg)^k} \left(\frac{f}{g}\right)^s &=
\left(
\frac{\partial}{\partial x_i} \left( \frac{h(s)}{(fg)^k} \right)
+ \frac{s \frac{\partial f}{\partial x_i} h(s)}{f (fg)^k}
- \frac{s \frac{\partial g}{\partial x_i} h(s)}{g(fg)^k}
\right)
\left(\frac{f}{g}\right)^s,
\end{aligned}
\]
for $h(s) \in \CC[x,s]$ and $k \in \ZZ_{\geq 0}$, cf.~\cite[Equation 1.4-1.5]{Tak23} and \cite[Definition 3.6]{AGLN21}. In other words, polynomials in $x,s$ act by the regular multiplication, while $\partial_i$ acts as the partial differentiation.

\begin{theorem}\label{thm:Takeuchi} {\rm (\cite[Theorem 1.1]{Tak23})}
Given a nonnegative integer $m$ there exists a nonzero polynomial $b_{\frac{f}{g},m}^{}(s) \in \CC[s]$ such that 
\begin{equation*}%\label{Eq:FuncEqTake}
b_{\frac{f}{g},m}^{}(s) \frac{1}{g^m} \left( \frac{f}{g} \right)^s \in
\sum_{k = 1}^\infty D[s] \frac{1}{g^m} \left( \frac{f}{g} \right)^{s+k}.
% \sum_{k = 1}^N P_k(s) \frac{1}{g^m} \left( \frac{f}{g} \right)^{s+k} = b_{\frac{f}{g},m}^{}(s) \frac{1}{g^m} \left( \frac{f}{g} \right)^s.
\end{equation*}
\end{theorem}

This result is a global version of the original Theorem by Takeuchi. In particular, he defines a local version of the
Bernstein-Sato polynomial, that is denoted by $b_{\frac{f}{g},m, x_0}^{}(s)$ for $x_0 \in \CC^n$. It is worth emphasizing the importance of $m$ in relation to the local Milnor monodromy of the rational function \cite{MR1647824,MR1734347}.
More precisely, consider the set $E_{\frac{f}{g}, x_0}$ of eigenvalues of the local Milnor fibration at $x_0$ and at points close to $x_0$. 
Moreover, assume that 
%$m$ is greater than or equal to twice the ambient dimension, i.e., if 
$m \geq 2n$ for the ambient dimension $n$, then in \cite[Theorem 1.4]{Tak23} Takeuchi proves, that 
\begin{equation}\label{eq:thm-eigen}
\left\{ \exp(2 \pi i \alpha) \mid b_{\frac{f}{g},m, x_0}(\alpha)=0 \right\} = E_{\frac{f}{g},x_0}.
\end{equation}
If $m < 2n$, then only the inclusion $\{\exp(2 \pi i \alpha) \mid b_{\frac{f}{g},m, x_0}(\alpha)=0\} \subseteq E_{\frac{f}{g}, x_0}$ holds in general.

We end this preliminary section with the definition of the Bernstein-Sato ideal of the pair $f,g$
that will be relevant for our purpose, see Theorem~\ref{thm:main-theorem}.
Its existence was first proven by Sabbah~\cite{Sab87}.

\begin{defi}\label{def:BS-ideal}
Given two non-constant polynomials $f,g \in \CC[x_1,\ldots,x_n]$, there exists a nonzero polynomial $b(s_1,s_2) \in \CC[s_1,s_2]$
such that
\[
b(s_1,s_2) f^{s_1} g^{s_2} = P(s_1,s_2) f^{s_1+1} g^{s_2+1},
\]
where $P(s_1,s_2) \in D[s_1,s_2] = D \otimes_\CC \CC[s_1,s_2]$.
The set of polynomials $b(s_1,s_2)$ satisfying this functional equation forms an ideal in $\CC[s_1,s_2]$, which is called the \emph{Bernstein-Sato ideal} of $f,g$, and is denoted by $B_{f,g}$.
\end{defi}

\section{Central saturation}\label{sec:sat}

The goal of this section is to recall an algorithm to compute the saturation with respect to $\CC[s]$ of a left ideal in~$D[s]$ starting from a system of generators of the ideal. The formal definition is given below.
We refer to e.g.~\cite[\S 8.7]{BW93} for a commutative version of this
method and to~\cite{HL21} for a more general approach including the noncommutative case.
%Viktor
Note, that the center of $D[s]$ is precisely the ring $\CC[s]$, and this holds for an arbitrary number of variables $s_1, \ldots, s_{\ell}$. Hence, the only two-sided ideals in $D[s]$ are those, extended from $\CC[s]$. Moreover, the important technique of \emph{saturation}, with which we deal below, can be treated by the much simpler \emph{central saturation}. %cf V+H

\begin{defi}
Given a left ideal $I \subset D[s]$, we define the $\CC[s]$-\emph{saturation} of $I$  as the left ideal %defined by
\[
\Sat_{\CC[s]}(I) := \{ P(s) \in D[s] \mid \exists q(s) \in \CC[s] \setminus \{0\}, \ q(s) P(s) \in I \}.
\]
This ideal contains $I$, i.e. $I \subseteq \Sat_{\CC[s]}(I)$, and we say that $I$ is $\CC[s]$-\emph{saturated} if the equality $I=\Sat_{\CC[s]}(I)$ takes place. Analogously, we can define this notion for a left ideal $I \subset D[s_1, \ldots, s_{\ell}]$ for $\ell \geq 2$.
\end{defi}

\begin{remark}
In a more general setting, saturation appears in the context of \emph{Ore localizations} of modules over a Noetherian domain. It is taken with respect to a multiplicatively closed Ore set (also called a denominator set). Strictly speaking, we should say $(\CC[s]\setminus\{0\})$-\emph{saturated}, but since we only deal with domains, we use the simpler term
$\CC[s]$-\emph{saturated}.
%Let $S$ denote, depending on the context, the set $\{s_1,\ldots,s_m\}$, where the cases $m=1,2$ deserve special attention.
The localization of $D[s]$ with respect to $\CC[s]\setminus\{0\}$ is denoted by $D(s)$ and can be seen being naturally isomorphic to $D \otimes_{\CC} \CC(s)$ via the extension of scalars. Thus the kernel of the homomorphism $D[s]/D[s]I \to D(s)/D(s)I$ of $D[s]$-modules is precisely $\Sat_{\CC[s]}(I)$. %cf V+H arxiv?
\end{remark}

The following result, namely Proposition~\ref{prop:computation-saturation}, collects some properties of the saturation
that will allow us to describe Algorithm~\ref{alg:sat} for computing it. Before that let us fix some notation.
We focus on the case of the single variable ${\ell}=1$. Let \[ \iota: D[s] \hookrightarrow D(s) \]
be the natural inclusion of $D[s]$-modules.
By abuse of notation the image of $P(s) \in D[s]$ is also denoted by $P(s) \in D(s)$.
If $I$ is a left ideal in $D[s]$, then $I^e := D(s) I$
denotes the \emph{extension} of~$I$. If $J$ is a left ideal in $D(s)$, then $J^c = J \cap D[s] = \iota^{-1}(J)$ denotes the \emph{contraction} of $J$.
In particular, $I^{ec}:=(I^{e})^c$ denotes the contraction of the extension of $I$. We can assume that a generating set of a left ideal in $D(s)$ we work with consists of elements from $D[s]$: this can always be achieved by clearing the denominators, which are polynomials in $s$.

Let us fix a monomial well-ordering $<_D$ on $D$. There are two natural orderings associated with $<_D$.
The first one is the ordering $<_{D(s)}$ on $D(s)$ that is defined as $<_D$, since $s$ belongs to the ground field of $D(s)$.
The second one is the ordering $<_{D[s]}$ on $D[s]$ defined as follows: 
%Extending scalars, this monomial ordering induces a monomial ordering in $D(s)$ that we denote by $<_{D(s)}$. In this context there is a monomial ordering in $D[s]$, denoted by $<_{D[s]}$, that plays an important role and that it is defined
\begin{equation}\label{monomial-ordering-Ds}
s^{i_1} x^{\alpha_1} \partial^{\beta_1} <_{D[s]} s^{i_2} x^{\alpha_2} \partial^{\beta_2}
\quad \Longleftrightarrow \quad
\begin{cases}
x^{\alpha_1} \partial^{\beta_1} <_D x^{\alpha_2} \partial^{\beta_2}, \\
\quad \text{or} \\[3pt]
x^{\alpha_1} \partial^{\beta_1} = x^{\alpha_2} \partial^{\beta_2} \ \text{and} \ i_1 < i_2.
\end{cases}
\end{equation}
%Using matrix notation, the relation between these monomial ordering is expressed as
The matrix associated with the latter is
\[
\begin{array}{ccc}
&& s \ \ \ x \ \partial \ \\
<_{D[s]} & = & \left(\begin{array}{c|c} 0 & <_D \\ \hline 1 & 0 \end{array}\right).
\end{array}
\]
This monomial ordering has the following properties: (i) it coincides with $<_D$, being restricted on $D$, and (ii) it behaves the same way as an ordering on $D(s)$ would.

Finally for $g \in D(s)$ let us denote by $LT_{<_{D(s)}}(g)$ the leading term
%of $g$ with respect to $<_{D(s)}$ 
and by $LC_{<_{D(s)}}(g)$ the leading coefficient of $g$ with respect to $<_{D(s)}$. Note that $LC_{<_{D(s)}}(g) \in \CC[s]$. 
Analogously, for $g \in D[s]$ one defines $LC_{<_{D[s]}}(g) \in \CC$. 
% Viktor, last line fixed: LC instead of LT

\begin{prop}\label{prop:computation-saturation}
Let $I$ be a left ideal in $D[s]$ and $J$ a left ideal in $D(s)$.
Choose any monomial ordering $<_D$ in $D$ and denote by $<_{D[s]}$
(resp.~$<_{D(s)}$) the monomial ordering induced by $<_D$ in $D[s]$
(resp.~$D(s)$) as above, see~\eqref{monomial-ordering-Ds}. The following properties hold:
\begin{enumerate}
\item $I^{ec} = \Sat_{\CC[s]}(I)$.
\item If $G_J$ is a Gr\"obner basis of $J$ with respect to $<_{D(s)}$ with $G_J \subset D[s]$
and $q(s) \in \CC[s]$ is the least common multiple of all $LC_{<_{D(s)}}(g)$ for $g \in G_J$, 
then $J^c = D[s]\langle G_J \rangle : q(s)^\infty$.
\item If $G_I$ a Gr\"obner basis of $I$ with respect to $<_{D[s]}$,
then $G_I$ is a Gr\"obner basis of $I^e$ with respect to $<_{D(s)}$.
\end{enumerate}
\end{prop}

\begin{remark}
The key point in the proof of (3) in the previous result relies on the fact that if $LT_{<_{D[s]}}(g) = s^{i} x^{\alpha} \partial^{\beta}$,
then $LT_{<_{D(s)}}(g) = x^{\alpha} \partial^{\beta}$. This is a consequence of the following property
\[
s^{i_1} x^{\alpha_1} \partial^{\beta_1} \leq_{D[s]} s^{i_2} x^{\alpha_2} \partial^{\beta_2}
\quad \Longrightarrow \quad 
x^{\alpha_1} \partial^{\beta_1} \leq_{D} x^{\alpha_2} \partial^{\beta_2}
\ (\text{or equivalently} \ x^{\alpha_1} \partial^{\beta_1} \leq_{D(s)} x^{\alpha_2} \partial^{\beta_2})
\]
that holds by definition of $<_{D[s]}$.
\end{remark}

The correctness of the following algorithm is an immediate consequence of Proposition~\ref{prop:computation-saturation}.

\begin{algorithm}\label{alg:sat}
	\SetKwInOut{Input}{Input}
	\SetKwInOut{Output}{Output}
	\caption{\textsc{Central saturation} \big(computes the $\CC[s]$-saturation of $I \subset D[s]$\big)}
	\Input{$I$ a left ideal in $D[s]$.}
	\Output{$\{P_1(s),\ldots,P_r(s)\} \subset D[s]$ a system of generators of $\Sat_{\CC[s]}(I)$.}
	%\KwData{$n \geq 0$}
	%\KwResult{$y = x^n$}
	%$y \gets 1$\;
	\Begin{
		Choose $<_D$ a monomial ordering in $D$\;
		$<_{D(s)}:=$ Monomial ordering in $D(s)$ induced by $<_D$ in $D(s)$\;
		$<_{D[s]}\,:=$ Monomial ordering in $D[s]$ defined in~\eqref{monomial-ordering-Ds}\;
		$G:=$ Gröbner basis of $I$ w.r.t.~$<_{D[s]}$ \tcc*[r]{by (3) a GB of $I^e$ w.r.t.~$<_{D(s)}$}
		$q(s):= \lcm \{ LC_{<_{D(s)}}(g) \mid g \in G\}$\;
		$J = I:q(s)^\infty$ \tcc*[r]{by (2) $J = D[s]\langle G \rangle:q(s)^\infty = I^{ec}	$}
		\Return $J$ \tcc*[r]{by (1) it equals $\Sat_{\CC[s]}(I)$}
	} 
\end{algorithm}

We will show an example of this algorithm in the next section, see Example~\ref{ex:Reiffen-xy-m3}.

\section{Annihilator of a rational function}\label{sec:annhilator}

Let $f,g \in \CC[x_1,\ldots,x_n]$ be two polynomials and $m \in \ZZ$, $m \geq 0$.
In this section we discuss the computation of the annihilator in $D[s]$
of $\frac{1}{g^m} \left( \frac{f}{g} \right)^s$ that we denote by
\begin{equation}\label{eq:ann-mero2}
I_m(s) := \Ann_{D[s]} \frac{1}{g^m} \left( \frac{f}{g} \right)^s.
\end{equation}
Towards the computation of \eqref{eq:ann-mero2}, we first compute the annihilator in $D[s_1,s_2]$ of
$f^{s_1} g^{s_2}$, see e.g.~\cite{OaTa99}, that we denote by
\[
I(s_1,s_2) = \Ann_{D[s_1,s_2]} f^{s_1} g^{s_2} \subset D[s_1,s_2].
\]
Then we notice that the ideal $I(s,-s-m)$ is in general not $\CC[s]$-saturated
and hence $I_m(s)$ and $I(s,-s-m)$ do not necessarily agree, see Proposition~\ref{prop:saturation}.
A sufficient condition for the saturation of $I(s,-s-m)$ to be equal to $I_m(s)$ is given
in Theorem~\ref{thm:main-theorem}. Finally, even if this sufficient condition is not satisfied,
we are still able to compute $I_m(s)$, see Section~\ref{sec:L-discussion}.

Let us start with the following result that we believe is well known for experts but
we include its proof for the sake of completeness.

\begin{prop}\label{prop:saturation}
The ideal $I_m(s)$ (resp.~$I(s_1,s_2)$) is $\CC[s]$-saturated (resp.~$\CC[s_1,s_2]$-saturated).
\end{prop}

\begin{proof}
We only prove the one-variable case, the other case is analogous. Consider $P(s) \in D[s]$ and $q(s) \in \CC[s]$ with
$q(s)\neq 0$ such that $q(s) P(s) \frac{1}{g^m} \left( \frac{f}{g} \right)^s = 0$. Recall that
\[
P(s) \frac{1}{g^m} \left( \frac{f}{g} \right)^s \in \CC\Big[x_1,\ldots,x_n,s,\frac{1}{fg}\Big] \left( \frac{f}{g} \right)^s.
\]
Thus $P(s) \frac{1}{g^m} \left( \frac{f}{g} \right)^s$ can be written as
$\frac{h(s)}{(fg)^k} \left( \frac{f}{g} \right)^s$ for some
polynomial $h(s) \in \CC[x_1,\ldots,x_n,s]$ and $k \geq 0$. Therefore
\[
q(s) P(s) \frac{1}{g^m} \left( \frac{f}{g} \right)^s
= \frac{q(s) h(s)}{(fg)^k} \left( \frac{f}{g} \right)^s
= 0
\]
and hence $\frac{q(s) h(s)}{(fg)^k} = 0$ or equivalently $\frac{h(s)}{(fg)^k} = 0$.
Then $P(s) \frac{1}{g^m} \left( \frac{f}{g} \right)^s = 0$ and the claim follows.
\end{proof}

\begin{remark}
Indeed, one can show more, while involving non-central localizations (though we do not go into details here). In the algebra $D[s]$,
the annihilator ideal of any $f^s$ and of $\left(\frac{f}{g}\right)^s$ is even $\CC[x_1,\ldots,x_n,s]$-saturated. This follows along the lines of the proof above, since no element of $\CC[x_1,\ldots,x_n,s]\setminus\{0\}$ annihilates any object which we are investigating.
\end{remark}

Recall that in general we can not expect the equality $I_m(s) = I(s,-s-m)$ to be true, see Example~\ref{ex:Reiffen-xy-m3}.
However, the relation given in Theorem~\ref{thm:main-theorem} inspired by the proof of \cite[Theorem 5.3.13]{SST00} holds. Note, that this proof actually holds for a more general statement than the mentioned Theorem 5.3.13, see \cite{Oaku12}.
% VL: I suggest a reformulation since this can scare a reader :)
%As a word of caution note that there is a typo in the statement of Theorem 5.3.13 in \cite{SST00} as acknowledged
%by one of the authors\footnote{\href{http://www.math.kobe-u.ac.jp/HOME/taka/SST/}{\texttt{http://www.math.kobe-u.ac.jp/HOME/taka/SST/}}}.
First we need to introduce some notation.

\begin{defi}\label{defi:C-condition}
Let $f,g \in \CC[x_1,\ldots,x_n]$ be two polynomials and $m \in \ZZ_{\geq 0}$.
Consider $B_{f,g} \subset \CC[s_1,s_2]$ the Bernstein-Sato ideal of $f$ and $g$, see Definition~\ref{def:BS-ideal}.
We say that $b(s_1,s_2) \in B_{f,g}$ satisfies the $\mathcal{C}_m$-\emph{condition} if
$b(s-i, -s-m-i) \neq 0$, for all $i \in \ZZ$, $i \geq 1$. Moreover we say that the ideal $B_{f,g}$ satisfies
the $\mathcal{C}_m$-\emph{condition} if there exists $b(s_1,s_2) \in B_{f,g}$ such that $b(s_1,s_2)$ satisfies
the $\mathcal{C}_m$-condition.
\end{defi}

\begin{remark}\label{rem:C-condition}
%Using Euclid's division algorithm in $R[s_2]$ with $R = \CC[s_1]$, 
%{\bf Viktor: $R[s_2]$ is NOT a Euclidean domain since $R=\CC[s_1]$ is not a field, what kind of division do you expect? You already need Gr\"obner bases for the appropriate division algorithm!}
Note that the following conditions are equivalent:
\begin{enumerate}
\item $b(s_1,s_2)$ does not satisfy the $\mathcal{C}_m$-condition,
\item there exists $i \geq 1$ such that $b(s-i,-s-m-i) = 0$,
\item there exists $i \geq 1$ such that $s_1 + s_2 + m + 2i$ divides $b(s_1,s_2)$ in $\CC[s_1,s_2]$.
\end{enumerate}
\end{remark}

\begin{theorem}\label{thm:main-theorem}
Assume that the Bernstein-Sato ideal $B_{f,g}$ satisfies the $\mathcal{C}_m$-condition.
Then
\[
I_m(s) = \Sat_{\CC[s]}(I(s,-s-m)).
\]
\end{theorem}

\begin{proof}
Clearly $I(s,-s-m) \subseteq I_m(s)$. Then $\Sat_{\CC[s]}(I(s,-s-m)) \subseteq \Sat_{\CC[s]}(I_m(s)) = I_m(s)$,
since $I_m(s)$ is $\CC[s]$-saturated, see Proposition~\ref{prop:saturation}.

For the other inclusion take $P(s) \in I_m(s)$.
We need to find $\tilde{P}(s_1,s_2) \in I(s_1,s_2)$ such that $\tilde{P}(s,-s-m) = q(s) P(s)$ for some
$q(s) \in \CC[s] \setminus \{0\}$.
In order to do this we use the functional equation associated with the element $b(s_1,s_2) \in B_{f,g}$
satisfying the $\mathcal{C}_m$-condition,
\begin{equation}\label{eq:BSfg}
Q(s_1,s_2) f^{s_1+1} g^{s_2+1} = b(s_1,s_2) f^{s_1} g^{s_2},
\end{equation}
where $Q(s_1,s_2) \in D[s_1,s_2]$. Let us write
\begin{equation}\label{eq:hfgk}
P(s_1) f^{s_1} g^{s_2}
= \frac{h(s_1,s_2)}{(fg)^k} f^{s_1} g^{s_2}
\ \in \ \CC \Big[ x_1,\ldots,x_n,s_1,s_2,\frac{1}{fg} \Big] f^{s_1} g^{s_2}
\end{equation}
for some $h(s_1,s_2) \in \CC[x_1,\ldots,x_n,s_1,s_2]$ and $k \geq 0$.
Applying the substitution $s_1 = s$, $s_2=-s-m$ to \eqref{eq:hfgk}, we obtain
\[
P(s) \frac{1}{g^m} \left( \frac{f}{g} \right)^s
= \frac{h(s,-s-m)}{(fg)^k} \frac{1}{g^m} \left( \frac{f}{g} \right)^s.
\]
Since $P(s) \frac{1}{g^m} \left( \frac{f}{g} \right)^s = 0$,
we get
\begin{equation}\label{eq:hzero}
h(s,-s-m) = 0.
\end{equation}
% or equivalently
% \[
% h(s_1,s_2) = (s_2 + s_1 + m) \tilde{h}(s_1,s_2)
% \]
% for some $\tilde{h}(s_1,s_2) \in \CC[x_1,\ldots,x_n,s_1,s_2]$.

If $k = 0$, then it is enough to choose $\tilde{P}(s_1,s_2) = P(s_1) - h(s_1,s_2)$ and $q(s)=1$.
Note that $\tilde{P}(s_1,s_2) \in I(s_1,s_2)$ by~\eqref{eq:hfgk} and moreover, $\tilde{P}(s,-s-m) = q(s) P(s)$ holds by~\eqref{eq:hzero}.

% If $k=1$, then~\eqref{eq:hfgk} and~\eqref{eq:BSfg} shifted by $1$ in $s_1$ and $s_2$ give rise to
% \[
% \begin{aligned}
% b(s_1-1,s_2-1) P(s_1) f^{s_1} g^{s_2}
% &= b(s_1-1,s_2-1) \frac{h(s_1,s_2)}{fg} f^{s_1} g^{s_2} \\
% &= h(s_1,s_2) b(s_1-1,s_2-1) f^{s_1-1} g^{s_2-1} \\[4pt]
% &= h(s_1,s_2) Q(s_1-1,s_2-1) f^{s_1} g^{s_2}.
% \end{aligned}
% \]
% Choose $\tilde{P}(s_1,s_2) = b(s_1-1,s_2-1) P(s_1) - h(s_1,s_2) Q(s_1-1,s_2-1) \in I_m(s_1,s_2)$
% and $\tilde{P}(s,-s-m) = b(s-1,s-m-1) P(s)$ holds by~\eqref{eq:hzero}. Note that $b(s-1,s-m-1) \neq 0$ by assumption.

If $k\geq 1$, then we multiply \eqref{eq:hfgk} by $b(s_1-k,s_2-k)$ and obtain
\[
b(s_1-k,s_2-k) P(s_1) f^{s_1} g^{s_2} = h(s_1,s_2) b(s_1-k,s_2-k) f^{s_1-k} g^{s_2-k}.
\]
Applying the functional equation~\eqref{eq:BSfg} to the right-hand side of the previous equation, we get
\[
b(s_1-k,s_2-k) P(s_1) f^{s_1} g^{s_2} = h(s_1,s_2) Q(s_1-k,s_2-k) f^{s_1-(k-1)} g^{s_2-(k-1)}.
\]
Iterating this process $k-1$ times and using~\eqref{eq:hfgk} and~\eqref{eq:BSfg} as above, we arrive at the following functional equation
\begin{equation}\label{eq:new-function-eq}
\tilde{b}(s_1,s_2) P(s_1) f^{s_1} g^{s_2} = h(s_1,s_2) R(s_1,s_2) f^{s_1} g^{s_2},
\end{equation}
\[
\mbox{where }
\tilde{b}(s_1,s_2) = \prod_{i=1}^k b(s_1-i, s_2-i) \mbox{ and } R(s_1,s_2) = Q(s_1-k,s_2-k) \cdot \ldots \cdot Q(s_1-1,s_2-1).
\]
Finally we choose
\[
\begin{aligned}
\tilde{P}(s_1,s_2) &= \tilde{b}(s_1,s_2) P(s_1) - h(s_1,s_2) R(s_1,s_2), \\
q(s) &= \tilde{b}(s,-s-m) = \prod_{i=1}^k b(s-i, -s-m-i).
\end{aligned}
\]
Therefore by~\eqref{eq:new-function-eq} we conclude that $\tilde{P}(s_1,s_2) \in I(s_1,s_2)$,
and because of~\eqref{eq:hzero} $\tilde{P}(s,-s-m) = \tilde{b}(s,-s-m) P(s)$ follows.
Moreover, $q(s) = \tilde{b}(s,-s-m) \neq 0$ because we have chosen $b(s_1,s_2)$ satisfying the $\mathcal{C}_m$-condition.
\end{proof}

\begin{remark}
There is a classical conjecture in $D$-module theory by Ucha-Enríquez \cite[Section 4.4.1]{And14} stating that $\Ann_{D[s]}(fg)^s = I(s,s)$. As in the proof of
Theorem~\ref{thm:main-theorem} we can show that if there exists $b(s_1,s_2) \in B_{f,g}$ such that $b(s_1-i,s_2-i) \neq 0$
for all $i \geq 1$, or equivalently $s_2-s_1$ does not divide $b(s_1,s_2)$, then
\[
\Ann_{D[s]}(fg)^s = \Sat_{\CC[s]}(I(s,s)).
\]
Since the condition $(s_2-s_1) \nmid b(s_1,s_2)$ is always accomplished for a certain choice of
$b(s_1,s_2) \in B_{f,g}$~\cite{Sab87I, Gyo93},
the previous conjecture says that the ideal $I(s,s)$ need not be saturated. We were not able to find a counterexample to this
conjecture but we have observed that for our goal the saturation is needed, as the following example shows.
\end{remark}

\begin{ex}\label{ex:Reiffen-xy-m3}
Let us consider $f = x^4 + y^5 + xy^4$, $g = xy$, and $m=3$.
Applying the method explained in Section~\ref{sec:sat} to $J := I(s,-s-3)$,
we obtain that $\Sat_{\CC[s]}(J) = J:(s-1)(11s-15)$. In fact, the factor $11s-15$
is superfluous since $J:11s-15 = J$ and then the saturation is simply $J:s-1$
that it is strictly bigger than $J$. For instance the element
\[
\begin{aligned}
& 704xy^3 \partial_x^2 - 576xy^3\partial_x \partial_y + 560y^4 \partial_x \partial_y - 432y^4 \partial_y^2 - 6875xy^2 \partial_x^2 + 5325xy^2 \partial_x \partial_y - 5500y^3\partial_x \partial_y \\
& + 12x^2y \partial_y^2 - 26xy^2\partial_y^2 + 4040y^3\partial_y^2 + 1728xy^2\partial_xs - 2880xy^2\partial_x - 1536y^3\partial_x s + 4496y^3\partial_x \\
& + 2592y^3\partial_y s - 5616y^3\partial_y - 11550xy\partial_x s + 24675xy\partial_x + 15125y^2\partial_x s - 44000y^2\partial_x + 12x^2\partial_y s \\
& + 48x^2\partial_y - 37xy\partial_y s - 108xy\partial_y - 20570y^2\partial_y s + 50665 y^2\partial_y - 3888y^2 s^2 + 15552y^2 s \\
& - 15120y^2 - 11xs^2 - 37xs - 12x + 26015ys^2 - 116050ys + 129615y
%& 704xy^3\partial_x^2-576xy^3\partial_x\partial_y + 560y^4\partial_x\partial_y - 432y^4\partial_y^2 - 6875xy^2\partial_x^2
%+ 5325xy^2\partial_x\partial_y - 5500y^3\partial_x\partial_y \\
%& +12x^2y\partial_y^2 - 26xy^2\partial_y^2 + 4040y^3\partial_y^2 + 1728xy^2\partial_xs - 1536y^3\partial_xs
%+ 2592y^3\partial_ys - 2880xy^2\partial_x \\
%& +4496y^3\partial_x - 5616y^3\partial_y - 11550xy\partial_xs + 15125y^2\partial_xs + 12x^2\partial_ys
%- 37xy\partial_ys - 20570y^2\partial_ys \\
%& -3888y^2s^2 + 24675xy\partial_x - 44000y^2\partial_x + 48x^2\partial_y - 108xy\partial_y + 50665y^2\partial_y
%+ 15552y^2s - 11xs^2 \\
%& +26015ys^2 - 15120y^2 - 37xs - 116050ys - 12x + 129615y
\end{aligned}
\]
is in $J:s-1$ but it is not in $J$. Theorem~\ref{thm:main-theorem} tells us that if the $B_{f,g}$ satisfies
the $\mathcal{C}_3$-condition, then $I_3(s) = J:s-1$. However, we were not able to compute the Bernstein-Sato
ideal of $f,g$.
\end{ex}

The rest of this section is devoted to study the case where $B_{f,g}$ does not satisfy the $\mathcal{C}_m$-condition,
see Definition~\ref{defi:C-condition}. We need to introduce some notation and a preliminary result.

\begin{notation}\label{notation-syz}
Let $\{Q_1(s), \ldots, Q_r(s)\} \subset D[s]$ be a system of generators of a left ideal $I \subset D[s]$ and take $P(s) \in D[s]$. Consider $\pi : D[s]^{r+1} \to D[s]$ the canonical  projection onto the first factor.
By abuse of notation we denote by $\syz_{D[s]}(P(s),I)$ the left submodule 
\[\syz_{D[s]}(P(s),Q_1(s),\ldots,Q_r(s)) \subset D[s]^{r+1}.
\]
This module can also be described as the kernel of the homomorphism of left $D[s]$-modules defined by
\[
\varphi_{P(s)}: D[s] \longrightarrow D[s]/I, \qquad Q(s) \mapsto Q(s) P(s) + I.
\]
\end{notation}

\begin{lemma}\label{lemmma:syz}
Using Notation~\ref{notation-syz}, the following holds for all $i \geq 0$ :
\[
I_{k-i}(s) = \pi\left(\syz_{D[s]}\left(f^i,I_k\left(s-i\right)\right)\right).
\]

\end{lemma}

\begin{proof}
Let us fix $\{Q_1(s),\ldots,Q_r(s)\}$ a system of generators of $I_k(s-i)$.
We need to prove that $I_{k-i}(s) = \pi(\syz_{D[s]}(f^i,Q_1(s),\ldots,Q_r(s)))$. 

Consider $P(s) \in I_{k-i}(s)$. By definition
\[
P(s) \frac{1}{g^{k-i}} \left( \frac{f}{g} \right)^s = 0,
\]
but $\frac{1}{g^{k-i}} \left( \frac{f}{g} \right)^s = f^i \frac{1}{g^k} \left( \frac{f}{g} \right)^{s-i}$
and hence $P(s) f^i \in I_k(s-i)$.
Then
\[
P(s) f^i = \sum_{i=1}^r R_i(s) Q_i(s),
\]
for some $R_i(s) \in D[s]$, $i=1,\ldots,r$. This implies
\[
(P(s),-R_1(s),\ldots,-R_r(s)) \in \syz_{D[s]} (f^i,Q_1(s),\ldots,Q_r(s)) \subset D[s]^{r+1}.
\]
The projection onto the first factor gives $P(s) \in \pi(\syz_{D[s]} (f^i,Q_1(s),\ldots,Q_r(s)))$.

To show the other inclusion note that all arguments hold in the other direction.
\end{proof}

\subsection{}\label{sec:L-discussion}
Recall that the goal of this section is to find an algorithm to compute $I_m(s)$ for all $m \in \ZZ$, $m \geq 0$.
In order to do this let us fix $b(s_1,s_2) \in \CC[s_1,s_2]$ a nonzero element of the Bernstein-Sato ideal $B_{f,g}$.
Consider the set
\[
L_{b(s_1,s_2)} := \Big\{ \ell \in \ZZ \mid s_1 + s_2 + \ell \mbox{ divides } b(s_1,s_2) \Big\}.
\]
It is finite (possibly empty) because $b(s_1,s_2) \neq 0$ cannot have infinitely many factors of the form $s_1 + s_2+\ell$.
By Remark~\ref{rem:C-condition}, $b(s_1,s_2)$ does not satisfy the $C_m$-condition if and only if
there exists $i \geq 1$ such that $m+2i \in L_{b(s_1,s_2)}$.

We distinguish three cases:

\begin{enumerate}
\renewcommand{\theenumi}{\alph{enumi}}
\item \label{L-item1} $L_{b(s_1,s_2)} = \emptyset$. Then $b(s_1,s_2)$ and $B_{f,g}$ satisfy the $\mathcal{C}_m$-condition for all $m \in \ZZ$, $m \geq 0$ and hence, by Theorem~\ref{thm:main-theorem}, $I_m(s)=\Sat_{\CC[s]}(I(s,-s-m))$.
\end{enumerate}
Now we define $e_{b(s_1,s_2)} = 1$ if $L_{b(s_1,s_2)} = \emptyset$ and $e_{b(s_1,s_2)} = \max L_{b(s_1,s_2)}$
otherwise.
\begin{enumerate}
\renewcommand{\theenumi}{\alph{enumi}}
\setcounter{enumi}{1}
\item \label{L-item2} $L_{b(s_1,s_2)} \neq \emptyset$ and $m \geq e_{b(s_1,s_2)} - 1$. Since $i \geq 1$, $m+2i \geq e_{b(s_1,s_2)} + 1$ and then
$m+2i$ cannot be in $L_{b(s_1,s_2)}$ because of the definition of $e_{b(s_1,s_2)}$. Therefore $B_{f,g}$ satisfies
the $\mathcal{C}_m$-condition and again by Theorem~\ref{thm:main-theorem}, $I_m(s)=\Sat_{\CC[s]}(I(s,-s-m))$.
\item \label{L-item3} $L_{b(s_1,s_2)} \neq \emptyset$ and $m < e_{b(s_1,s_2)} - 1$. In this case $b(s_1,s_2)$ does not satisfy
the $\mathcal{C}_m$-condition and we cannot apply Theorem~\ref{thm:main-theorem} directly.
Instead we use Lemma~\ref{lemmma:syz} with $i = e_{b(s_1,s_2)}-1-m > 0$ and $k = m+i$ to obtain
\begin{equation*}%\label{eq:Ims}
I_{m}(s) = \pi(\syz_{D[s]}(f^{e_{b(s_1,s_2)}-1-m},I_{e_{b(s_1,s_2)}-1}(s-e_{b(s_1,s_2)}+1+m))).
\end{equation*}
Note that the computation of the previous ideal $I_{e_{b(s_1,s_2)-1}}(s)$ is treated in~\eqref{L-item2}.
\end{enumerate}

\begin{remark}\label{rk:one-step}
Lemma~\ref{lemmma:syz} with $i=1$ and $k = e_{b(s_1,s_2)}-j$ allows us to compute $I_{e_{b(s_1,s_2)}-j-1}(s)$ from
$I_{e_{b(s_1,s_2)}-j}(s-1)$ for all $j \geq 1$. This gives an alternative method to case \eqref{L-item3} when
$e_{b(s_1,s_2)}-1-m$ is too large.
\end{remark}

\begin{remark}\label{rk:choice-b}
The choice of $b(s_1,s_2)$ from the system of generators of $B_{f,g}$ with the smallest $e_{b(s_1,s_2)}$
improves the efficiency of the computation of $I_m(s)$ along the previous discussion.
\end{remark}

We summarize this method in the following algorithm.

\begin{algorithm}\label{alg:ann}
\SetKwInOut{Input}{Input}
\SetKwInOut{Output}{Output}
\caption{\textsc{Annihilator} \Big(computes the annihilator of $\frac{1}{g^m} \left( \frac{f}{g} \right)^s$ in $D[s]$\Big)}
\Input{$f,g$ two polynomials in $\CC[x_1,\ldots,x_n]$; \\ $m \geq 0$ an integer.}
\Output{$\{P_1(s),\ldots,P_r(s)\} \subset D[s]$ a system of generators of $I_m(s) = \Ann_{D[s]} \frac{1}{g^m} \left( \frac{f}{g} \right)^s$.}
%\KwData{$n \geq 0$}
%\KwResult{$y = x^n$}
%$y \gets 1$\;
\Begin{
	$I(s_1,s_2):=$ Annihilator of $f^{s_1} g^{s_2}$ in $D[s_1,s_2]$\;
	$B_{f,g}:= \langle b_1(s_1,s_2), \ldots, b_r(s_1,s_1) \rangle$, Bernstein-Sato ideal of $f,g$ in $\CC[s_1,s_2]$\;
	$\varepsilon := \min \left\{ e_{b_i(s_1,s_2)} \mid i=1,\ldots,r \right\} - 1$
	\tcc*[r]{see Remark~\ref{rk:choice-b}}
	$\gamma := \varepsilon - m$\;
	\eIf{$\gamma \leq 0$}
	{
		\Return $\Sat_{\CC[s]} (I(s,-s-m))$ \tcc*[r]{cases \eqref{L-item1} and \eqref{L-item2} above}
	}
	{
		%$J := \Sat_{\CC[s]}(I(s-\gamma,-s-m))$\;
		%\Return $\pi(\syz_{D[s]}(f^{\gamma},J))$\;
		$J(s) := \Sat_{\CC[s]}(I(s,-s-\varepsilon))$ \tcc*[r]{equals $I_\varepsilon(s)$ due to case~\eqref{L-item2}}
		\Return $\pi(\syz_{\CC[s]}(f^\gamma,J(s-\gamma)))$ \tcc*[r]{case~\eqref{L-item3}}
	}
}
\end{algorithm}

\begin{ex}\label{ex:ann}
Let us consider $f = x^2 + y^2$ and $g = xy$. The purpose of this example is to compute $I_m(s)$ for all $m \in \ZZ$, $m \geq 0$.
The annihilator of $f^{s_1}g^{s_2}$ in $D[s_1,s_2]$ is generated by two operators
\begin{equation}\label{eq:PQ}
\begin{aligned}
% VL, esthetics: multiply P by -1
%P(s_1,s_2) &= -x \partial_x - y \partial_y + 2s_1 + 2s_2, \\
P(s_1,s_2) &= x \partial_x + y \partial_y - 2s_1 - 2s_2, \\
Q(s_1,s_2) &= xy^2 \partial_x - x^2y \partial_y + x^2s_2 - y^2 s_2,
\end{aligned}
\end{equation}
that is, $I(s_1,s_2) = D[s_1,s_2] \left\langle P(s_1,s_2), Q(s_1,s_2) \right\rangle$.
The Bernstein-Sato ideal of $f,g$ is generated by a single element
\begin{equation}\label{eq:BS-ideal-x2y2}
\begin{aligned}
b(s_1,s_2)  = \
& (s_1+1) (s_2+1) \cdot \\
& (s_1+s_2+1) (s_1+s_2+2) (s_1+s_2+3)  \cdot  \\
& (2s_1+2s_2+3) (2s_1+2s_2+5).
\end{aligned}
\end{equation}

Using the algorithm described in Section \ref{sec:sat} one checks that $I(s,-s-m)$ is $\CC[s]$-saturated for all $m \in \ZZ$, $m \geq 0$.
Following the notation in Section~\ref{sec:L-discussion} (see Notation~\ref{notation-syz} and Algorithm~\ref{alg:ann}) one obtains
\[
L_{b(s_1,s_2)} = \{1,2,3\}, \qquad e_{b(s_1,s_2)} = 3, \qquad \varepsilon = 2, \qquad \gamma = 2 - m.
\]
Hence $I_m(s) = \Sat_{\CC[s]}(I(s,-s-m)) = I(s,-s-m)$ for all $m \geq 2$.

It remains to compute $I_0(s)$ and $I_1(s)$. For $m=1$ we use
\[
I_1(s) = \pi (\syz_{D[s]}(f,I_2(s-1)))
\]
and obtain three generators $P(s,-s-1)$, $Q(s,-s-1)$, $R(s)$, where
\[
R(s) = y^2 \partial_x \partial_y - x y \partial_y^2 - y \partial_x s - x \partial_y s + y \partial_x - 2x \partial_y.
\]
Note that the new operator $R(s)$ is not in $I(s,-s-1)$ and then $I_1(s) \supsetneq I(s,-s-1)$.

For $m=0$ there are two ways to compute $I_0(s)$, namely, either $\pi (\syz_{D[s]}(f^2,I_2(s-2)))$ or $\pi (\syz_{D[s]}(f,I_1(s-1)))$,
see Remark~\ref{rk:one-step}. The ideal $I_0(s)$ is generated by three operators $P(s,-s)$, $Q(s,-s)$, $T(s)$, where
\[
T(s) = y \partial_x^4 \partial_y - y \partial_y^5 - \partial_x^4 s + 2 \partial_x^2 \partial_y^2 s - \partial_y^4 s - 4 \partial_x^2 \partial_y^2 - 4 \partial_y^4.
\]
Once again one verifies that $I_0(s) \supsetneq I(s,-s)$.
\end{ex}

\section{Bernstein-Sato polynomial of rational functions}\label{sec:BS}

Let $f,g \in \CC[x_1,\ldots,x_n]$ be two polynomials and $m \in \ZZ$, $m \geq 0$.
Takeuchi \cite{Tak23} showed that for a big enough $N$ there exists a nonzero polynomial $b(s) \in \CC[s]$
and $P_1(s),\ldots,P_N(s) \in D[s]$ such
\begin{equation}\label{eq:functional-eq}
b(s) \frac{1}{g^m} \left( \frac{f}{g} \right)^s =
\sum_{k = 1}^N P_k(s) \frac{1}{g^m} \left( \frac{f}{g} \right)^{s+k}.
\end{equation}
The set of polynomials $b(s)$ verifying a functional equation of the type \eqref{eq:functional-eq} is an ideal in $\CC[s]$.
We denote by $b_{\frac{f}{g},m}^{(N)}(s)$ its monic generator.
In this section we present an algorithm to compute $b_{\frac{f}{g},m}^{(N)}(s)$ for a fixed $N \geq 1$. Note that expression~\eqref{eq:functional-eq} can alternatively be written as
\[
\left( b_{\frac{f}{g},m}^{(N)}(s) - \sum_{k = 1}^N P_k(s) \frac{f^k}{g^k} \right) \frac{1}{g^m} \left( \frac{f}{g} \right)^s = 0.
\]
Cleaning the denominators in the first factor of the previous expression to get an element in $D[s]$ we obtain
an equivalent equation
\[
\left( b_{\frac{f}{g},m}^{(N)}(s) g^N - \sum_{k = 1}^N P_k(s) f^k g^{N-k} \right) \frac{1}{g^{m+N}} \left( \frac{f}{g} \right)^s = 0.
\]
In Section~\ref{sec:annhilator}, we have provided a method to compute the annihilator
\[
I_m(s) := \Ann_{D[s]} \frac{1}{g^m} \left( \frac{f}{g} \right)^s.
\]
Using this notation we have
\begin{equation}\label{eq:BernsteinGN}
b_{\frac{f}{g},m}^{(N)} (s) g^N \in J \quad \mbox{where} \quad
J = I_{m+N}(s) + \sum_{k = 1}^N D[s] f^k g^{N-k}.
\end{equation}
Following the ideas in the proof of Lemma~\ref{lemmma:syz} and recalling Notation~\ref{notation-syz},
one finally gets the expression
\begin{equation}\label{eq:intersection-Cs}
%\left( ( D[s] J \cap \CC[x_1,\ldots,x_n,s] ) : g^N \right) \cap \CC[s] = \CC[s] \left\langle b_{\frac{f}{g},m}^N(s) \right\rangle,
\pi(\syz_{D[s]}(g^N,J)) \cap \CC[s] = \Big\langle b_{\frac{f}{g},m}^{(N)}(s) \Big\rangle.
\end{equation}

Note that~\eqref{eq:intersection-Cs} provides a method for computing the Bernstein-Sato polynomial of $\frac{f}{g}$, once a system of generator of $I_{m+N}(s) \subset D[s]$ has been obtained, see Section~\ref{sec:annhilator}.

\begin{ex}
Consider $f = x^2 + y^2$ and $g = xy$ as in Example \ref{ex:ann}, where the annihilator of $\frac{1}{g^m}\left(\frac{f}{g}\right)^s$
was computed for all $m \geq 0$. Recall that
\[
I_m(s) = D[s] \langle P(s,-s-m), Q(s,-s-m) \rangle,
\]
for all $m \geq 2$, where $P(s_1,s_2)$ and $Q(s_1,s_2)$ are the operators given in~\eqref{eq:PQ}.

Now we set $m=1$ and $N=1$. Following~\eqref{eq:BernsteinGN} and~\eqref{eq:intersection-Cs},
we first need to consider the ideal $D[s]g + J$, where $J$ is generated by $P(s,-s-2),Q(s,-s-2)$ and $f$. The syzygy module of the tuple $(g,P(s,-s-2),Q(s,-s-2),f)$ is generated by the rows of the following matrix
\[
\begin{pmatrix}
y\partial_x-x\partial_y&  y^2&0&       y\partial_y+1 \\        
x\partial_x+y\partial_y+2&xy&0&       0 \\  
x^2+y^2&    0&  0&       -xy \\ 
-y\partial_x+x\partial_y& x^2&0&       x\partial_x+1 \\
y^2 \partial_x-xy\partial_y-2xs-2x &     0&  -y&      ys+y \\
xy\partial_x+y^2\partial_y+2ys+4y &     0&  -x&      -xy\partial_y-xs-2x \\
%0&          0&  -x^2-y^2& xy^2\partial_x-x^2y\partial_y-x^2s+y^2s-2x^2+2y^2
\end{pmatrix}.
\]
The ideal generated by the elements of the first column can also be generated by
\[
 ys+y, \
 xs+x, \
 y\partial_x-x\partial_y, \
 x\partial_x+y\partial_y+2, \
 x^2+y^2.
\]
and then the intersection with $\CC[s]$ is zero, that is, $b_{\frac{f}{g},1}^{(1)}(s) = 0$.
This phenomenon has to do with the fact that the generator of $B_{f,g}$
verifies that $b(s,-s-m-2) = b(s,-s-3) = 0$, see~\eqref{eq:BS-ideal-x2y2},
cf.~Oaku's remark in \cite[p.~717]{Tak23}. In this case at least two terms are needed in Theorem~\ref{thm:Takeuchi}, i.e.,
two operators $P_1(s)$ and $P_2(s)$ in the functional equation \eqref{eq:functional-N} are nontrivial.
Note that the number of operators needed to obtain the Bernstein-Sato polynomial can be very large,
see Example~\ref{ex:largeN}.

For $m=1$ and $N=2$ the functional equation is given by
\[
(s+1) \frac{1}{g} \left( \frac{f}{g} \right)^s =
P_1(s) \frac{1}{g} \left( \frac{f}{g} \right)^{s+1} +
P_2(s) \frac{1}{g} \left( \frac{f}{g} \right)^{s+2},
\]
where
\[
\begin{aligned}
P_1(s) &= \frac{1}{4}y^4\partial_x^3\partial_y - \frac{1}{4}xy^3\partial_x^2\partial_y^2 + \frac{1}{4}y^4\partial_x\partial_y^3
- \frac{1}{4}xy^3\partial_y^4 - \frac{1}{2}y^3\partial_x^3s + \frac{1}{2}xy^2\partial_x^2\partial_ys - \frac{1}{4}y^3\partial_x^3 \\
& \quad + \frac{1}{2}xy^2\partial_x^2\partial_y + \frac{3}{4}y^3\partial_x\partial_y^2 - \frac{1}{2}xy^2\partial_y^3 + 2y^2\partial_x\partial_ys + \frac{15}{4}y^2\partial_x\partial_y + \frac{1}{4}xy\partial_y^2 + \frac{1}{2}y\partial_xs
+ \frac{5}{4}y\partial_x, \\[5pt]
P_2(s) &= -\frac{1}{4}y^3\partial_x^2\partial_ys - \frac{1}{4}y^3\partial_y^3s - \frac{1}{2}y^3\partial_x^2\partial_y - \frac{1}{2}y^3\partial_y^3 - \frac{1}{2}y^2\partial_y^2s - y^2\partial_y^2 + \frac{1}{4}y\partial_ys + \frac{1}{2} y\partial_y.
\end{aligned}
\]
Hence the Bernstein-Sato polynomial of $\frac{f}{g}$ for $m=1$ is $s+1$.
\end{ex}

Alternatively, if one knows in advance that $b_{\frac{f}{g},m}^{(N)}(s)$ is not zero, there is a better approach for the computation of the Bernstein-Sato polynomial. Indeed, \eqref{eq:BernsteinGN} says that
\[
\overline{b_{\frac{f}{g},m}^{(N)} (s) g^N} = \overline{0} \in \frac{D[s]}{J}.
\]
Then the set
\[
\left\{ \overline{1 \cdot g^N}, \overline{s \cdot g^N}, \overline{s^2 \cdot g^N}, \ldots \right\}
\]
is linearly dependent in $\frac{D[s]}{J}$ and the linear relation with minimal degree in $s$ is precisely the Bernstein-Sato polynomial of $\frac{f}{g}$.
This approach is much more efficient that the one given by formula~\eqref{eq:intersection-Cs}. By contrast,
if $b_{\frac{f}{g},m}^{(N)}(s)$ is zero, then the set $\big\{ \overline{s^i \cdot g^N} \big\}_{i \geq 0}$
is linearly independent.

We finish this section presenting an algorithm for computing $b_{\frac{f}{g},m}^{(N)}(s)$ assuming it is not zero.
Its correctness follows from the previous discussion.

\begin{algorithm}\label{alg:BSmero}
\SetKwInOut{Input}{Input}
\SetKwInOut{Output}{Output}
\caption{\textsc{Bernstein-Sato} \Big(computes $b_{\frac{f}{g},m}^{(N)}(s)$ assuming it is not zero\Big)}
\Input{$f,g$ two polynomials in $\CC[x_1,\ldots,x_n]$;\\ $m \geq 0$ and $N \geq 1$ integers;\\ $<$ a monomial order on $D[s]$.}
\Output{$b_{\frac{f}{g},m}^{(N)}(s) \in \CC[s]$, a monic generator of the ideal defined by~\eqref{eq:functional-eq}.}
\Begin{
	$I:=I_{m+N}(s)$ \tcc*[r]{use Algorithm~\ref{alg:ann}}
	$J:= I + D[s] \langle fg^{N-1}, f^2g^{N-2},\ldots,f^N \rangle$ \tcc*[r]{see \eqref{eq:BernsteinGN}}
	$G:=$ Gr\"obner basis of $J$ with respect to $<$ \;
	$P_0:=NF(g^N,G)$ \tcc*[r]{the remainder of the division of $g^N$ by $G$}
	\If{$P_0=0$}
	{
		\Return $b=1$\;
	}
	$k:=1$\;
	\While{true}
	%{$(1=1)$} 
	{
		$P_k = NF(s^kg^N,G)$\;
		\If{$\exists a_0,\ldots,a_{k} \in \QQ$ (with $a_k=1$) such that $a_0 P_0 + \ldots + a_k P_k = 0$}
		{
			\Return $b = a_0 + a_1 s + \ldots + a_k s^k$ \tcc*[r]{by the assumption $b_{\frac{f}{g},m}^{(N)}(s) \neq 0$}
		}
		$k := k+1$\;
	}
}
\end{algorithm}

\section{Examples}\label{sec:examples}

\begin{ex}\label{ex:largeN}
Let us consider $f = x^2 + y^3$ and $g = xy$. For all $m \geq 0$ the ideal $I(s,-s-m)$ is $\CC[s]$-saturated
and the Bernstein-Sato ideal satisfies the $\mathcal{C}_m$-conditions.
Then, by Theorem~\ref{thm:main-theorem}, $I_m(s) = I(s,-s-m)$ for all $m\geq 0$.

Running Algorithm~\ref{alg:BSmero} for $m = 0$ one obtains
\[
\begin{aligned}
b_{\frac{f}{g},0}^{(1)} (s) &= (s+1)(s+5)(s+7), \\
b_{\frac{f}{g},0}^{(N)} (s) &= (s+1)(s+5), \quad N = 2,\ldots,5, \\
b_{\frac{f}{g},0}^{(6)} (s) &= s+1.
\end{aligned}
\]
In this case we need six differential operators to achieve the Bernstein-Sato polynomial
of $\frac{f}{g}$ which is $s+1$, that is, $b_{\frac{f}{g},0} (s) = b_{\frac{f}{g},0}^{(N)} (s)$ for all $N \geq 6$.

On the other hand, for $m=1$ Algorithm~\ref{alg:BSmero} yields
\[
\begin{aligned}
b_{\frac{f}{g},1}^{(1)} (s) &= s(s+1)(s+2), \\
b_{\frac{f}{g},1}^{(N)} (s) &= s(s+1), \quad N = 2,\ldots,20.
\end{aligned}
\]
However, we could not verify whether $b_{\frac{f}{g},1}^{(N)}(s) = s+1$ for $N$ big enough as in case $m=0$.
\end{ex}

In the next example we will use the eigenvalues of the monodromy of a rational function \cite{MR1647824,MR1734347}
to determine the Bernstein-Sato polynomial.

\begin{ex}
Let $f=x^2+y^3$, $g = x$, and $m=4$. The Bernstein-Sato ideal satisfies the $\mathcal{C}_m$-condition and
the ideal $I(s,-s-m)$ is $\CC[s]$-saturated for all $m\geq 0$. According to Theorem~\ref{thm:main-theorem}
$I_m(s) = I(s,-s-m)$ for all $m \geq 0$.
Algorithm~\ref{alg:BSmero} for $N=1$ produces
\begin{equation}\label{eq:tres-roots}
b_{\frac{f}{g},4}^{(1)}(s) = \Big(s+1\Big) \Big(s-\frac{7}{3}\Big) \Big(s-\frac{5}{3}\Big).
\end{equation}
Note that not all the roots are negative by contrast with the classical case \cite{Kas76}.

We plan to use the equality \eqref{eq:thm-eigen} to show that the polynomial in~\eqref{eq:tres-roots}
is in fact the Bernstein-Sato polynomial of $\frac{f}{g}$. A resolution $\pi: X \to (\CC^2,0)$ of the
singularity of the quotient can be achieved with four blow-ups.
The final situation is illustrated in Figure~\ref{fig:resolution} where the multiplicities
of $f$ and $g$ along the exceptional divisors are shown.
By \cite[Theorem 1]{MR1647824} the local monodromy zeta functions of the rational function are
\[
Z_{{\rm mon},\frac{f}{g}}(t)_0 = \frac{1-t}{1-t^3} = \frac{1}{t^2+t+1},
\qquad Z_{{\rm mon},\frac{f}{g}}(t)_a = 1-t, \quad a \in V(f) \setminus \{0\},
\]
and then the eigenvalues are all the roots of unity of order $3$.
Equality \eqref{eq:thm-eigen} implies that $b_{\frac{f}{g},4}(s)$ has exactly $3$ different
roots modulo $\ZZ$. Since $b_{\frac{f}{g},4}(s)$ divides $b_{\frac{f}{g},4}^{(N)}(s)$
which in turn divides $b_{\frac{f}{g},4}^{(1)}(s)$ for all $N\geq 1$ and the latter
has already $3$ different roots modulo $\ZZ$, it follows that they all must be equal.

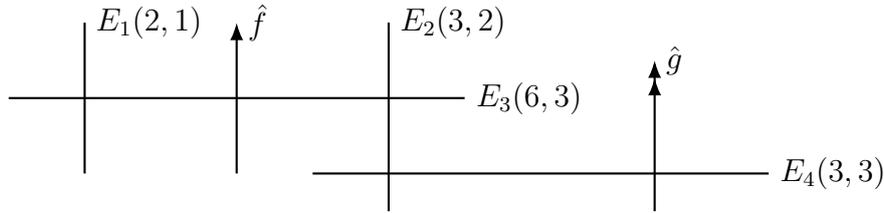
\begin{figure}[ht]
\begin{tikzpicture}[scale=1]
%\draw[help lines,step=1] (0,0) grid (6,3);
\draw[thick] (0,1.5) -- (6,1.5) node[right]{$E_3(6,3)$}; 
\draw[thick] (1,0.5) -- (1,2.5) node[right]{$E_1(2,1)$};
\draw[-{Latex[]},thick] (3,0.5) -- (3,2.5) node[right]{$\hat{f}$};
\draw[thick] (5,0) -- (5,2.5) node[right]{$E_2(3,2)$};
\draw[thick] (4,0.5) -- (10,0.5) node[right]{$E_4(3,3)$};
\draw[-{Latex[] Latex[]},thick] (8.5,0) -- (8.5,2) node[right]{$\hat{g}$};
\end{tikzpicture}
\caption{Resolution of $\frac{x^2+y^3}{x}$.}
\label{fig:resolution}
\end{figure}

The resolution help us also to get a bound for the roots of the Bernstein-Sato polynomial
according to \cite[Theorem 4.1]{Tak23}. More precisely, all the roots of the polynomial~\eqref{eq:tres-roots}
are contained in the set
%Viktor: \ell is nicer than l
\[
B_{\frac{f}{g},m}^{\pi} = \left\{ m - \frac{1}{3} - \ell, \ m-\frac{2}{3}-\ell, \ 2m-1-\ell \ \Big| \ \ell \in \ZZ, \, \ell \geq 0 \right \}.
\]
Notice that the elements of the form $m-\frac{1}{3}-\ell$ and $m-\frac{2}{3}-\ell$ come from $E_3$ while $2m-1-\ell$ comes from $E_2$.

Heuristically we have observed that
\begin{equation*}
b_{\frac{f}{g},m}^{(N)}(s) = \Big(s+1\Big) \Big(s-m+\frac{5}{3}\Big) \Big(s-m+\frac{7}{3}\Big).
\end{equation*}
However, for $m =0,1,2,3$, we cannot use \eqref{eq:thm-eigen} and conclude that the Bernstein-Sato polynomial
of the rational function has really three factors.

It is worth mentioning that the weak monodromy conjecture for plane rational functions was proven in~\cite{GL14}.
In this case the local topological zeta function at the origin of $\frac{f}{g}$ is
\[
Z_{\textrm{top},\frac{f}{g}}(s)_o = \frac{11s + 15}{12 \big(s+1\big)\big(s+\frac{5}{3}\big)}.
\]
In the context of the monodromy conjecture this example suggests that it is natural to relate the poles of
$Z_{\textrm{top},\frac{f}{g}}(s)_o$ with the roots of $b_{\frac{f}{g},0}(s)$, that is, when $m=0$.
\end{ex}

\begin{ex}
Consider $f = x^6 + y^6 + 2zx^3y^3$ and $g = z^2$. The Bernstein-Sato ideal of $f,g$ was computed in~\cite{BO10}
\[
B_{f,g} = b(s_1,s_2) \langle 3s_1+5, 2s_2+3 \rangle,
\]
where
\[
\begin{aligned}
b(s_1,s_2) = \, & (s_1+1)^2 (2s_1+1) (2s_1+3) (3s_1+1) (3s_1+2) \cdot \\
& (3s_1+4) (6s_1+5) (6s_1+7) (s_2+1) (2s_2+1)
\end{aligned}
\]
Then $B_{f,g}$ satisfies the $\mathcal{C}_m$-condition for all $m \geq 0$
and by Theorem~\ref{thm:main-theorem},
\[
I_m(s) = \Sat_{\CC[s]}(I(s,-s-m)), \ \forall m \geq 0.
\]

Using the methods from Section~\ref{sec:sat}, we have checked that
the ideal $I(s,-s-m)$ is $\CC[s]$-saturated for all $m \geq 2$ while $I(s,-s-1)$ is not.
In fact,
\[
I_1(s) = \Sat_{\CC[s]}(I(s,-s-1)) = I(s,-s-1): \Big(s+\frac{1}{2}\Big) \supsetneq I(s,-s-1).
\]

Algorithm~\ref{alg:BSmero} for $m = 0$ and $N = 1,\ldots,4$ yields
\[
b_{\frac{f}{g},0}^{(N)}(s) = 
\Big(s+1\Big)^2
\Big(s+\frac{1}{2}\Big)
\Big(s+\frac{3}{2}\Big)
\Big(s+\frac{1}{3}\Big)
\Big(s+\frac{2}{3}\Big)
\Big(s+\frac{4}{3}\Big)
\Big(s+\frac{5}{6}\Big)
\Big(s+\frac{7}{6}\Big).
\]
Note that $s+\frac{1}{2}$ remains as a factor of $b_{\frac{f}{g},0}^{(N)}(s)$, although it plays a role in the saturation of $I(s,-s-1)$.
\end{ex}

\bibliographystyle{amsplain}
\bibliography{ref_BS}

\end{document}